\documentclass{amsart}

\newtheorem{lemma}{Lemma}[section]
\newtheorem{theorem}[lemma]{Theorem}
\newtheorem{remark}[lemma]{Remark}

\newtheorem{coro}[lemma]{Corollary}
\newtheorem{definition}[lemma]{Definition}
\newtheorem{example}[lemma]{Example}

\title{I. U. Bronshtein's Conjecture for Monotone Nonautonomous
Dynamical Systems}

\author{David ~Cheban}
\address[D. Cheban]{%
State University of Moldova\\ Faculty of Mathematics and
Informatics\\ Department of Mathematics\\ A. Mateevich Street 60\\
MD--2009 Chi\c{s}in\u{a}u, Moldova} \email[D.
Cheban]{cheban@usm.md, davidcheban@yahoo.com}

\date{\today}

\begin{document}
\begin{abstract}
In this paper we study the problem of almost periodicity of
solutions for dissipative differential equations (Bronshtein's
conjecture). We give a positive answer to this conjecture for
monotone almost periodic systems of differential/difference
equations.
\end{abstract}

\maketitle

\section{Introduction}\label{S1}

\textbf{I. U. Bronshtein's conjecture \cite[ChIV,p.273]{bro75}.}
If an equation
\begin{equation}\label{eqBC1}
x'=f(t,x)\ \ (f\in C(\mathbb R\times \mathbb R^{d},\mathbb R^d))
\end{equation}
with right hand side (Bohr) almost periodic in $t$ satisfies the
conditions of uniform positive stability and positive
dissipativity, then it has at least one (Bohr) almost periodic
solution.

\begin{remark} 1. If $d\le 3$, then the positive answer to this conjecture follows
from the results of V. V. Zhikov \cite[ChII]{Zhi69} (see also
\cite[ChVII]{Lev-Zhi} and \cite[ChIV]{bro75}).

2. Even for scalar equations ($d=1$) as was shown by A. M. Fink
and P. O. Frederickson \cite{FF} (see  also \cite[ChXII]{Fin}),
dissipation (without uniform positive stability) does not imply
the existence of almost periodic solutions.
\end{remark}

The aim of this paper is studying the problem of existence of
Levitan/Bohr almost periodic (respectively, almost automorphic,
recurrent and Poisson stable) solutions for dissipative
differential equation (\ref{eqBC1}), when the second right hand
side is monotone with respect to spacial variable. The existence
at least one quasi periodic (respectively, Bohr almost periodic,
almost automorphic, recurrent, pseudo recurrent, Levitan almost
periodic, almost recurrent, Poisson stable) solution of
(\ref{eqBC1}) is proved under the condition that every solution of
equation (\ref{eqBC1}) is positively uniformly Lyapunov stable.

The paper is organized as follows.

In Section 2 we collected some notions and facts from the theory
of dynamical systems (the both autonomous and nonautonomous) which
we use in this paper: Poisson stable motions and functions,
cocycles, skew-product dynamical systems, monotone non-autonomous
dynamical systems, Ellis semigroup.

Section 3 is dedicated to the study  the global attractors of
cocycles, when the phase space of driving system is noncompact.the
structure of the $\omega$-limit set of noncompact semitrajectory
for autonomous and nonautonomous dynamical systems.

In Section 4 we formulate I. U. Bronshtein's conjecture for
general nonautonomous dynamical systems. The positive answer for
monotone nonautonomous dynamical systems is given (Theorem
\ref{thBM1}, Corollary \ref{corPM1} and Remark \ref{remBM1}).

Section 5 is dedicated to the applications of our general results
for differential (Theorems \ref{thA1} and \ref{thA2}) and
difference (Theorems \ref{thD1} and \ref{thD2}) equations.

\section{Some general properties of autonomous and nonautonomous dynamical systems}\label{S2}

In this section we collect some notions and facts from the
autonomous and non-autonomous dynamical systems \cite{Che_2001}
(see also, \cite[Ch.IX]{Che_2015}) which we will use below.

\subsection{Cocycles.}

Let $Y$ be a complete metric space, let $\mathbb
R:=(-\infty,+\infty)$, $\mathbb Z :=\{0,\pm 1, \pm 2, \ldots \}$,
$\mathbb T =\mathbb R$ or $\mathbb Z$,  $\mathbb T _{+}=\{t\in
\mathbb T |\quad t \ge 0 \}$ and $\mathbb T_{-}=\{t \in \mathbb T|
\quad t \le 0 \}$. Let $(Y ,\mathbb T, \sigma )$ be an autonomous
two-sided dynamical system on $Y $ and $E$ be a real or complex
Banach space with the norm $\vert \cdot \vert $.

\begin{definition}\label{def2.1}(Cocycle on the state space $E$ with the base
$(Y ,\mathbb T ,\sigma)$). The triplet $\langle E, \phi ,(Y ,
\mathbb T, \sigma )\rangle (\mbox{or  briefly}\quad \phi) $ is
said to be a cocycle (see, for example, \cite{Che_2015} and
\cite{Sel_71}) on the state space $ E$ with the base $(Y,\mathbb
T,\sigma )$ if the mapping $\phi : \mathbb T_{+} \times  Y \times
E \to E $ satisfies the following conditions:
\begin{enumerate}
\item $\phi (0,y,u)=u $ for all $u\in E$ and $y\in Y$; \item $\phi
(t+\tau ,y ,u)=\phi (t,\phi (\tau ,u,y),\sigma(\tau,y))$ for all $
t, \tau \in \mathbb T _{+},u \in E$ and $ y \in Y$; \item the
mapping $\phi $ is continuous.
\end{enumerate}
\end{definition}

\begin{definition}\label{def2.2} (Skew-product dynamical system). Let $ \langle E ,
\phi,(Y,\mathbb T,\sigma)\rangle $ be a cocycle on $ E, X:=E\times
Y$ and $\pi $ be a mapping from $ \mathbb T _{+} \times X $ to $X$
defined by equality $\pi =(\phi ,\sigma)$, i.e., $\pi
(t,(u,y))=(\phi (t,\omega,u),\sigma(t,y))$ for all $ t\in \mathbb
T _{+}$ and $(u,y)\in E\times Y$. The triplet $ (X,\mathbb T _{+},
\pi)$ is an autonomous dynamical system and it is called
\cite{Sel_71} a skew-product dynamical system.
\end{definition}

\begin{definition}\label{def2.3} (Nonautonomous dynamical system.) Let
$ \mathbb T_{1}\subseteq \mathbb T _{2} $ be two subsemigroup of
the group $\mathbb T, (X,\mathbb T _{1},\pi ) $ and $(Y ,\mathbb T
_{2}, \sigma )$ be two autonomous dynamical systems and $h: X \to
Y$ be a homomorphism from $(X,\mathbb T _{1},\pi )$ to $(Y ,
\mathbb T _{2},\sigma)$ (i.e., $h(\pi(t,x))=\sigma(t,h(x)) $ for
all $t\in \mathbb T _{1} $, $ x \in X $ and $h $ is continuous),
then the triplet $\langle (X,\mathbb T _{1},\pi ),$ $ (Y ,$
$\mathbb T _{2},$ $ \sigma ),h \rangle $ is called (see
\cite{bro75} and \cite{Che_2015}) a nonautonomous dynamical
system.
\end{definition}

\begin{example} \label{ex2.4} (The nonautonomous dynamical system generated by cocycle $\phi$.)
Let $\langle E, \phi ,(Y ,\mathbb T,\sigma)\rangle $ be a cocycle,
$(X,\mathbb T _{+},\pi ) $ be a skew-product dynamical system
($X=E\times Y, \pi =(\phi,\sigma)$) and $h= pr _{2}: X \to Y ,$
then the triplet $\langle (X,\mathbb T _{+},\pi ),$ $(Y ,\mathbb
T,\sigma ),h \rangle $ is a nonautonomous dynamical system.
\end{example}

\subsection{Some general facts about nonautonomous dynamical
systems.}

In this subsection we give some general facts about nonautonomous
dynamical systems without proofs. The more details and the proofs
the readers can find in \cite{Che_2001} (see also
\cite[Ch.IX]{Che_2015}).

\begin{definition}\label{def3.2}  Let $(X,h,Y)$ be a fiber
space, i.e., $X$ and $Y$ be two metric spaces and $ h: X \to Y$ be
a homomorphism from $X$ into $Y$. The subset $ M \subseteq X$ is
said to be conditionally precompact
\cite{Che_2001},\cite[Ch.IX]{Che_2015}, if the preimage $h^{-1}(Y
')\bigcap M $ of every precompact subset $ Y '\subseteq Y$ is a
precompact subset of $X$. In particularly $M_{y}=h^{-1}(y)\bigcap
M$ is a precompact subset of $X_{y}$ for every $y\in Y $. The set
$ M $ is called conditionally compact if it is closed and
conditionally precompact.
\end{definition}

\begin{definition}\label{def3.11} Let
$\langle E,\phi, (Y,\mathbb T,\sigma)\rangle  \quad (respectively,
(X,\mathbb T_{+},\pi))$ be a cocycle (respectively, onesided
dynamical system). The continuous mapping $\nu :\mathbb T \to E $
(respectively, $\gamma :\mathbb T \to X$) is called an entire
trajectory of cocycle $\phi $ (respectively, of dynamical system
$(X,\mathbb T_{+},\pi)$) passing through the point $(u,y)\in
E\times Y$ (respectively, $x\in X$) for $t=0$ if $ \phi (t,\nu
(s),\sigma(s,y))=\nu (t+s) \ \mbox{and} \ \nu (0)=u $
(respectively, $\pi (t,\gamma (s))=\gamma (t+s)\ \mbox{and} \
\gamma(0)=x$) for all $ t\in \mathbb T_{+} $and $s\in \mathbb T.$
\end{definition}

Denote by  $\Phi_{x}$ the family of all entire trajectories of
$(X,\mathbb T_{+},\pi)$ passing through the point $x\in X$ at the
initial moment $t=0$ and $\Phi:=\bigcup \{\Phi_{x}:\ x\in X\}$.

\subsection{Bohr/Levitan almost periodic, almost automorphic,
recurrent and Poisson stable motions}

\begin{definition}\label{def4.2}
A number $\tau \in \mathbb S$ is called an $\varepsilon >0$ shift
of $x$ (respectively, almost period of $x$), if $\rho
(x\tau,x)<\varepsilon $ (respectively, $\rho (x(\tau
+t),xt)<\varepsilon$ for all $t\in \mathbb S$).
\end{definition}

\begin{definition}\label{def14.3}
A point $x \in X $ is called almost recurrent (respectively, Bohr
almost periodic), if for any $\varepsilon >0$ there exists a
positive number $l$ such that at any segment of length $l$ there
is an $\varepsilon$ shift (respectively, almost period) of point
$x\in X$.
\end{definition}

\begin{definition}\label{def4.4}
If the point $x\in X$ is almost recurrent and the set
$H(x):=\overline{\{xt\ \vert \ t\in \mathbb T\}}$ is compact, then
$x$ is called recurrent.
\end{definition}

Denote by $\mathfrak N_{x}:=\{\{t_n\}:\ \{t_n\}\subset \mathbb T \
\mbox{such that}\ \{\pi(t_n,x)\}\to x\ \mbox{as}\ n\to \infty\}$.

\begin{definition}
A point $x\in X$ of the dynamical system $(X,\mathbb T, \pi)$ is
called Levitan almost periodic \cite{Lev-Zhi}, if there exists a
dynamical system $(Y,\mathbb T,\sigma)$ and a Bohr almost periodic
point $y\in Y$ such that $\mathfrak N_{y}\subseteq \mathfrak
N_{x}.$
\end{definition}

\begin{definition} A point $x\in X$ is called stable in the sense
of Lagrange $(st.$L$)$, if its trajectory
$\Sigma_{x}:=\Phi\{\pi(t,x)\ :\ t\in \mathbb T\}$ is relatively
compact.
\end{definition}

\begin{definition} A point $x\in X$ is called almost automorphic
in the dynamical system $(X,\mathbb T,\pi),$ if the following
conditions hold: \begin{enumerate}\item $x$ is st.$L$; \item the
point $x\in X$ is Levitan almost periodic.
\end{enumerate}
\end{definition}

\begin{definition}\label{defPC1} A point $x_0\in X$ is called
\cite{Sch85,sib}
\begin{enumerate}
\item[-] pseudo recurrent if for any $\varepsilon >0$,
$t_0\in\mathbb T$ and $p\in \Sigma_{x_0}$ there exist numbers
$L=L(\varepsilon,t_0)>0$ and $\tau =\tau(\varepsilon,t_0,p)\in
[t_0,t_0+L]$ such that $\tau \in \mathfrak T (p,\varepsilon))$;
\item[-] pseudo periodic (or uniformly Poisson stable) if for any
$\varepsilon >0$, $t_0\in\mathbb T$ there exists a number  $\tau
=\tau(\varepsilon,t_0)>t_0$ such that $\tau \in \mathfrak T
(p,\varepsilon))$ for any $p\in \Sigma_{x_0}$; \item[-] Poisson
stable in the positive (respectively, negative) direction if for
any $\varepsilon >0$ and $l>0$ (respectively, $l<0$) there exists
a number $\tau >l$ (respectively, $\tau < l$) such that
$\rho(\pi(\tau,x_0),x_0)<\varepsilon.$ The point $x_0\in X$ is
called Poisson stable if it is stable (in the sense of Poisson) in
the both directions.
\end{enumerate}
\end{definition}

\begin{remark}\label{remPR_1} 1. Every pseudo periodic point is
pseudo recurrent.

2. If $x\in X$ is pseudo recurrent, then
\begin{enumerate}
\item[-] it is Poisson stable; \item[-] every point $p\in H(x)$ is
pseudo recurrent; \item[-] there exist pseudo recurrent points for
which the set $H(x_0)$ is compact but not minimal
\cite[ChV]{scher72}; \item[-] there exist pseudo recurrent points
which are not almost automorphic (respectively, pseudo periodic)
\cite[ChV]{scher72}.
\end{enumerate}
\end{remark}

\subsection{B. A. Shcherbakov's principle of comparability of
motions by their character of recurrence}

In this subsection we will present some notions and results stated
and proved by B. A. Shcherbakov \cite{scher72}-\cite{Sch85}.

Let $(X,\mathbb T_1,\pi)$ and $(Y,\mathbb T_2,\sigma)$ be two
dynamical systems.

\begin{definition}\label{defShc1}
A point $x\in X$ is said to be comparable with $y\in Y$ by the
character of recurrence, if for all $\varepsilon
>0$ there exists a $\delta =\delta(\varepsilon)>0$ such that every
$\delta$--shift of $y$ is an $\varepsilon$--shift for $x$, i.e.,
$d(\sigma(\tau,y),y)<\delta$ implies
$\rho(\pi(\tau,x),x)<\varepsilon$, where $d$ (respectively, $\rho$)
is the distance on $Y$ (respectively, on $X$).
\end{definition}

\begin{theorem}\label{thShc2.1} Let $x$ be comparable with $y\in Y$. If
the point $y\in Y$ is stationary (respectively, $\tau$--periodic,
Levitan almost periodic, almost recurrent, Poisson stable), then the
point $x\in X$ is so.
\end{theorem}

\begin{definition}
A point $x\in X$ is called \textit{uniformly comparable with $y\in
Y$ by character of recurrence}, if for all $\varepsilon >0$ there
exists a $\delta =\delta(\varepsilon)>0$ such that every
$\delta$--shift of $\sigma(t,y)$ is an $\varepsilon$--shift for
$\pi(t,x)$ for all $t\in\mathbb T$, i.e.,
$d(\sigma(t+\tau,y),\sigma(t,y))<\delta$ implies
$\rho(\pi(t+\tau,x),x)<\varepsilon$ for all $t\in \mathbb T$ (or
equivalently, $d(\sigma(t_1,y),\sigma(t_2,y))<\delta$ implies
$\rho(\pi(t_1,x),\pi(t_2,x))<\varepsilon$ for all $t_1,t_2\in
\mathbb T$).
\end{definition}

Denote by $\mathfrak{M}_{x}:=\{\{t_{n}\}\subset\mathbb{R} \ :\
\mbox{such that}\ \{\pi(t_{n},x)\}$ converges $ \}$.

\begin{definition} A point $x\in X$ is said \cite{Che_1977},\cite[ChII]{Che_2009} to be strongly
comparable by character of recurrence with the point $y\in Y$, if
$\mathfrak M_{y}\subseteq \mathfrak M_{x}$.
\end{definition}

\begin{definition} A point $y\in Y$ is said to be:
\begin{enumerate}
\item  stable in the sense of Lagrange in the positive direction
(respectively, stable in the sense of Lagrange) if the set
$H^{+}(y):=\overline{\{\sigma(t,y)|\ t\in\mathbb T_{+}\}}$
(respectively, $H(y):=\overline{\{\sigma(t,y)|\ t\in\mathbb T\}}$)
is compact; \item Poisson stable in the positive direction if
$x\in \omega_{x}$, where
$$
\omega_{x}:=\bigcap_{t\ge 0}\overline{\bigcup_{\tau \ge
t}\pi(\tau,x)}\ .
$$
\end{enumerate}
\end{definition}

\begin{theorem}\label{thShc5} Let $X$ be a complete metric space.
If the point $x$ uniformly comparable by character of recurrence
with $y$, then $\mathfrak M_{y}\subseteq \mathfrak M_{x}$.
\end{theorem}

\begin{theorem}\label{thShc6} Let $y$ be stable in the sense of
Lagrange. The inclusion $\mathfrak M_{y}\subseteq \mathfrak M_{x}$
takes place, if and only if the point $x$ is stable in the sense of
Lagrange and the point $x$ uniformly comparable by character of
recurrence with $y$.
\end{theorem}

\begin{theorem}\label{tS2} Let $X$ and $Y$ be two complete metric spaces,
the point $x$ be uniformly comparable with $y\in Y$ by the character
of recurrence. If the point $y\in Y$ is recurrent (respectively,
almost periodic, almost automorphic, uniformly Poisson stable), then
so is the point $x\in X$.
\end{theorem}

\subsection{Monotone Nonautonomous Dynamical Systems}

Let $\mathbb R_{+}^{d}:=\{x\in \mathbb R^{d}:\ $ such that $x_i\ge
0$ ($x:=(x_1,\ldots,x_n)$) for any $i=1,2,\ldots,d\}$ be the cone
of nonnegative vectors of $\mathbb R^d$. By $\mathbb R_{+}^{d}$ on
the space $\mathbb R^{d}$ is defined a partial order. Namely:
$u\le v$ if $v-u\in \mathbb R_{+}^{d}$. Let $K \subset \mathbb
R^{d}$ be a compact subset of $\mathbb R^{d}$, and for each $1\le
i \le d,$ define $\alpha_{i}(K):=\min\{x_i|\ x=(x_1,\ldots,x_d)\in
K\}$ and $\beta_{i}(K):=\max\{x_i|\ x=(x_1,\ldots,x_d)\in K\}$.
Then $\alpha (K):=(\alpha_1(K),\ldots,\alpha_d(K))$ and $\beta
(K):=(\beta_1(K),\ldots,\beta_d(K))$ are the greatest lower bound
(\emph{infimum}) and least upper bound (\emph{supremum}) of with
respect to the order on $\mathbb{R}^{d}$, respectively.

\begin{definition}\label{defMSP} Let $\langle \mathbb R^{d},\varphi, (Y,\mathbb T,\sigma)\rangle
$ be a cocycle and $\langle (X,\mathbb T_{+},\pi), (Y,\mathbb
T,\sigma),h\rangle$ be a nonautonomous dynamical system associated
by cocycle $\varphi$ (i.e., $X:=\mathbb R^d\times Y$, $\pi
=(\varphi,\sigma)$ and $h:=pr_2 :X\to Y$). The cocycle $\varphi$
is said to be monotone if $u_1\le u_2$ implies
$\varphi(t,u_1,y)\le \varphi(t,u_2,y)$ for any $t>0$ and $y\in Y$.
\end{definition}

Recall that a forward orbit $\{\pi(t,x_0)\ t\ge 0\}$ of
non-autonomous dynamical systems $\langle (X,\mathbb T_{+},\pi),
(Y,\mathbb T,\sigma),h\rangle$ is said to be uniformly stable if
for any $\varepsilon > 0$, there is a $\delta
=\delta(\varepsilon)>0$ such that
$\rho(\pi(t_0,x_0),\pi(t_0,x_0))<\delta$ implies
$d(\pi(t,x_0),\pi(t,x_0))<\varepsilon $ for every $t\ge t_0$.

Below we will use the following assumptions:
\begin{enumerate}
\item[(C1)] For every compact subset $K\subset X$ and $y\in Y$ the
set $K_{y}:=h^{-1}(y)\bigcap K$ has both the greatest lower bound
(g.l.b.) $\alpha_{y}(K)$ and the least upper bound (l.u.b.)
$\beta_{y}(K)$. \item[(C2)] For every $x\in X$, the
semi-trajectory $\Sigma^{+}_{x}:=\{\pi(t,x):\ t\ge 0\}$ is
conditionally precompact and its $\omega$-limit set $\omega_{x}$
is positively uniformly stable.

\item[(C3)] The non-autonomous dynamical system $\langle
(X,\mathbb T_{+},\pi), (Y,\mathbb T,\sigma),h\rangle $ generated
by cocycle $\varphi$ is monotone.
\end{enumerate}


\begin{lemma}\label{lAPS2}\cite{CL_2017} Assume that $(C1)$--$(C3)$ hold, $x_0\in X$ such that
$\omega_{x_0}$ is positively uniformly stable. Let $K:=\omega_{x_0}$
be fixed and $y_0:=h(x_0)$. Then if $q\in \omega_{q}\subseteq
\omega_{y_0}$, $\alpha_{q}:=\alpha_{q}(K)$,
$K^{1}:=\omega_{\alpha_{q}}$, then the set
$K_{q}^{1}:=\omega_{\alpha_{q}}\bigcap X_{q}$ (respectively,
$\omega_{\beta_{q}}\bigcap X_{q}$) consists a single point
$\gamma_{q}$ (respectively, $\delta_{q}$), i.e.,
$K_{q}^{1}=\{\gamma_{q}\}$ (respectively, $\{\delta_{q}\}$).
\end{lemma}

\begin{theorem}\label{thM1} \cite{CL_2017} Assume that $(C1)$--$(C3)$ hold, $x_0\in X$ such that
$\omega_{x_0}$ is positively uniformly stable and $y_0:=h(x_0)$.
Then the following statements hold:
\begin{enumerate}
\item if $y_0\in \omega_{y_0}$, then the point $\gamma_{y_0}$
(respectively, $\beta_{y_0}$) is comparable by character of
recurrence with $y _0$ and \item
\begin{equation}\label{eqP0}
\lim\limits_{n\to
\infty}\rho(\pi(t,\alpha_{y_0}),\pi(t,\gamma_{y_0}))=0\ .\nonumber
\end{equation}
\end{enumerate}
\end{theorem}

\begin{coro}\label{corP1} Under the conditions $(C1)-(C3)$ if the point $y_0$ is $\tau$-periodic
(respectively, Levitan almost periodic, almost recurrent, almost
automorphic, recurrent, Poisson stable), then:
\begin{enumerate}
\item the point $\gamma_{y_0}$ is so;
\item the point $\alpha_{y_0}$ is asymptotically $\tau$-periodic
(respectively, asymptotically Levitan almost periodic,
asymptotically almost recurrent, asymptotically almost automorphic,
asymptotically recurrent, asymptotically Poisson stable).
\end{enumerate}
\end{coro}

\begin{definition}\label{defP1} A point $x_0\in X$ is said to be:
\begin{enumerate}
\item[-] pseudo recurrent \cite{Shc_1964,Sch85,sib} if for any $\varepsilon
>0,\ p\in \Sigma_{x_0}:=\{\pi(t,x_0):\ t\in\mathbb T\}$ and $t_0\in
\mathbb T$ there exists $L=L(\varepsilon,t_0)>0$ such that
\begin{equation}\label{eqLP1.0}
B(p,\varepsilon)\bigcap \pi([t_0,t_0+L],p)\not=\emptyset ,\nonumber
\end{equation}
where $B(p,\varepsilon):=\{x\in X:\ \rho(p,x)<\varepsilon\}$ and
$\pi([t_0,t_0+L],p):=\{\pi(t,p):\ t\in [t_0,t_0+L]\}$; \item[-]
uniformly Poisson stable \cite{Beb_1940} (or pseudo peiodic
\cite[ChII,p.32]{Bohr_I1947}) if for arbitrary $\varepsilon >0$
and $l>0$ there exists a number $\tau >l$ such that
$\rho(\pi(t+\tau,x),\pi(t,x))<\varepsilon $ for any $t\in \mathbb
T$.
\end{enumerate}
\end{definition}

\begin{remark}\label{remP2} 1. Every recurrent (respectively, uniformly Poisson stable) point is pseudo
recurrent. The inverse statement, generally speaking, is not true.

2. If $x_0\in X$ is a pseudo recurrent point, then $p\in \omega_{p}$
for  any $p\in H(x_0)$.

3. If $x_0$ is a Lagrange stable point and $p\in\omega_{p}$ for any
$p\in H(x_0)$, then the point $x_0$ is pseudo recurrent.
\end{remark}

\begin{definition}\label{defSP1} A point $x\in X$ is said to be
strongly Poisson stable if $p\in \omega_{p}$ for any $p\in H(x)$.
\end{definition}

\begin{remark}\label{remP3}  Every pseudo
recurrent point is strongly Poisson stable. The inverse statement,
generally speaking, is not true.
\end{remark}

\begin{theorem}\label{thM2} \cite{CL_2017} Assume that $(C1)$--$(C3)$ hold, $x_0\in X$ and
$y_0:=h(x_0)\in Y$ is strongly Poisson stable.
Then the following statements hold:
\begin{enumerate}
\item the point $\gamma_{y_0}$ (respectively, $\delta_{y_0}$) is
strongly comparable by character of recurrence with $y _0$ and
\begin{equation}\label{eqPM3}
\lim\limits_{t\to
+\infty}\rho(\pi(t,\alpha_{y_0}),\pi(t,\gamma_{y_0}))=0 .\nonumber
\end{equation}
\end{enumerate}
\end{theorem}

\begin{coro}\label{corP2_1} Under the conditions $(C1)-(C3)$ if the point $y_0$ is $\tau$-periodic
(respectively, quasi periodic, Bohr almost periodic, recurrent,
pseudo recurrent and Lagrange stable), then:
\begin{enumerate}
\item the point $u_{y_0}$ is so;
\item the point $\alpha_{y_0}$ is asymptotically $\tau$-periodic
(respectively, asymptotically quasi periodic, asymptotically
Bohr almost periodic, asymptotically recurrent,
pseudo recurrent).
\end{enumerate}
\end{coro}

\begin{remark}\label{remP01} 1. If the point $y_0$ is recurrent (in
the sense of Birkhoff), then Corollary \ref{corP2_1} coincides with
the results of the work  of J. Jiang and X.-Q. Zhao \cite{JZ_2005}.

2. In the works of B. A. Shcherbakov \cite{Shc_1962}-\cite{Shc_1964}, \cite[ChV, Example 5.2.1]{scher72}
 were constructed examples
of pseudo recurrent and Lagrange stable motions which are not
recurrent (in the sense of Birkhoff).
\end{remark}

\section{Global Attractors of Cocycles}\label{S3}

Let $W$ be a complete metric space.

\begin{definition}\label{def2.7.1}
The family $ \{ I_{y} \ \vert \   y \in Y \} \ ( I_{y} \subset W
)$ of nonempty compact subsets $W$ is called (see, for example,
\cite{Arn} and \cite{FlaSchm95a}) a compact pullback
attractor\index{pullback attractor} (uniform pullback attractor
\index{uniform pullback attractor}) of a cocycle $ \varphi $, if
the following conditions hold:
\begin{enumerate}
\item the set $ I := \bigcup \{ I_{y} \ \vert \   y \in Y \} $ is
relatively compact; \item the family $ \{ I_{y} \ \vert \   y \in
Y \} $ is invariant with respect to the cocycle $ \varphi $, i.e.
$ \varphi (t, I_{y}, y ) = I_{\sigma (t,y )} $ for all $ t \in
\mathbb T_{+} $ and $ y \in Y $; \item for all $ y \in Y $
(uniformly in $y \in Y$) and $ K \in C(W) $
$$
\lim \limits _{t \to + \infty } \beta (\varphi (t, K,
y_{\sigma(-t,y)} ), I_{y}) =0 ,
$$
where $ \beta (A,B): = \sup \{ \rho (a,B) : a \in  A \} $ is a
semi-distance of Hausdorff.
\end{enumerate}
\end{definition}

\begin{definition}\label{def2.7.2}
The family $ \{ I_{y} \ \vert \   y \in Y \} ( I_{y} \subset W ) $
of nonempty compact subsets is called a compact (forward) global
attractor of the cocycle $\varphi,$ if the following conditions
are fulfilled:
\begin{enumerate}
\item the set $ I := \bigcup \{ I_{y} \ \vert \   y \in Y \} $ is
relatively compact; \item the family $ \{ I_{y} \ \vert \   y \in
Y \} $ is invariant with respect to the cocycle $ \varphi $; \item
the equality
$$
\lim \limits _{t \to + \infty } \sup \limits _{y \in Y } \beta
(\varphi (t,K,y ),I)=0
$$
holds for every $K\in C(W)$.
\end{enumerate}
\end{definition}

Let $M \subseteq W$ and
\begin{equation}\label{eq2.7.3}
\omega_y(M):=\bigcap_{t\ge 0} \overline{\bigcup_{\tau\ge t}
\varphi(\tau,M,\sigma(-\tau,y))} \ \nonumber
\end{equation}
for any $y \in Y$.

\begin{lemma}\label{l2.7.1}\cite[ChII]{Che_2015} The following statements hold:
\begin{enumerate}
\item The point $ p \in \omega_y(M)$ if and only if there exit
$t_n \to + \infty $ and $ \{x_n \} \subseteq M$ such that $
p=\lim\limits_{n \to + \infty} \varphi(t_n,x_n,\sigma(-t_n,y));$
\item $ U(t,y) \omega_y(M) \subseteq \omega_{\sigma(t,y)}(M) $ for
all $ y \in Y $ and $ t \in \mathbb T_+$, where
$U(t,y):=\varphi(t,\cdot,y);$ \item for any point $ w \in
\omega_y(M)$ the motion $\varphi(t,w,y)$ is defined on $ \mathbb S
$; \item if there exits a nonempty compact $ K \subset W $ such
that
$$
\lim_{t \to + \infty} \beta (\varphi(t,M,\sigma(-t,y)),K)=0,
$$
then $\omega_y(M) \ne \emptyset $, is compact,
\begin{equation}\label{eq2.7.5}
\lim\limits_{t \to + \infty} \beta (\varphi(t,M,\sigma(-t,y)),
   \omega _y(M))=0 \nonumber
\end{equation}
and
\begin{equation}\label{eq2.7.6}
U(t,y)\omega_y(M)=\omega_{\sigma(t,y)}(M) \nonumber
\end{equation}
for all $y \in Y$ and  $t \in \mathbb T_{+}$ .
\end{enumerate}
\end{lemma}

\begin{definition}\label{def2.7.3}
A cocycle $\varphi $ over $(Y,\mathbb T,\sigma)$ with the fiber
$W$ is said to be compactly dissipative, if there exits a nonempty
compact $K \subseteq W$ such that
\begin{equation}\label{eq2.7.8}
\lim_{t \to + \infty} \sup \{ \beta (U(t,y)M,K) \ \vert \  y \in Y
\}=0
\end{equation}
for any  $M \in C(W)$.
\end{definition}

\begin{theorem}\label{t2.7.3}\cite[ChII]{Che_2015}
Let $\langle W,\varphi,(Y,\mathbb T,\sigma)\rangle $ be compactly
dissipative and $K$ be the nonempty compact subset of $W$
appearing in the equality (\ref{eq2.7.8}), then:
\begin{enumerate}
\item[1.] $I_y=\omega_y(K) \ne \emptyset $, is compact, $ I_y
\subseteq K$ and
$$
\lim_{t \to + \infty} \beta(U(t,\sigma(-t,y))K,I_y)=0
$$
for every $y \in Y$; \item[2.] $U(t,y)I_y=I_{y t} $ for all $y \in
Y$ and $t \in \mathbb T_+$; \item[3.]
\begin{equation}\label{eq2.7.10}
\lim_{t \to + \infty}\beta(U(t,\sigma(-t,y))M,I_y)=0 \nonumber
\end{equation}
for all $M \in C(W)$ and $ y \in Y$ ; \item[4.] the set $I$ is
relatively compact, where $I:=\cup \{ I_y \ \vert \ y \in Y \} $.
\end{enumerate}
\end{theorem}

\begin{theorem}\label{t2.7.3F}
Let $\langle W,\varphi,(Y,\mathbb T,\sigma)\rangle $ be compactly
dissipative and $K$ be the nonempty compact subset of $W$
appearing in the equality (\ref{eq2.7.8}), then the family of
subsets $\{I_{y}|\ y\in Y\}$ is a maximal family possessing the
properties 2.--4.
\end{theorem}
\begin{proof} Let $\{I_{y}^{'}|\ y\in Y\}$ be a family of subsets
possessing properties 2.--4. Denote by $I^{'}:=\bigcup
\{I_{y}^{'}|\ y\in Y\}$ and $M:=\overline{I^{'}}$, where by ar is
denoted the closure of $I^{'}$. Since $M\in C(W)$, then for
arbitrary $\varepsilon >0$ and $y\in Y$ there exists a positive
number $L(\varepsilon,y)$ such that
\begin{equation}\label{eqF2}
U(t,\sigma(-t,y))M\subseteq B(I_{y},\varepsilon) \nonumber
\end{equation}
for any $t\ge L(\varepsilon,y)$. Note that $I_{y}^{'} =
U(t,\sigma(-t,y))I_{\sigma(t,y)}^{'}\subseteq
U(t,\sigma(-t,y))M\subseteq B(I_{y},\varepsilon)$. Since
$\varepsilon$ is an arbitrary positive number we obtain
$I^{'}_{y}\subseteq I_{y}$ for any $y\in Y$.
\end{proof}

\begin{definition}\label{defGA} Let $\langle W,\varphi,(Y,\mathbb T,\sigma)\rangle
$ be compactly dissipative, $K$ be the nonempty compact subset of
$W$ appearing in the equality (\ref{eq2.7.8}) and
$I_{y}:=\omega_{y}(K)$ for any $y\in Y$. The family of compact
subsets $\{I_y|\ y\in Y\}$ is said to be a Levinson center
(compact global attractor) of nonautonomous (cocycle) dynamical
system $\langle W,\varphi,(Y,\mathbb T,\sigma)\rangle $.
\end{definition}

\begin{remark}\label{remGA} According to Theorem \ref{t2.7.3F} by
definition \ref{defGA} is defined correctly the notion Levinson
center (compact global attractor) for nonautonomous (cocycle)
dynamical system $\langle W,\varphi,(Y,\mathbb T,\sigma)\rangle $.
\end{remark}

\begin{coro}\label{corF1}  Let $\langle W,\varphi,(Y,\mathbb T,\sigma)\rangle
$ be compactly dissipative nonautonomous dynamical system,
$\{I_{y}|\ y\in Y\}$ be its Levinson center and $\gamma :\mathbb
T\mapsto W$ be a relatively compact full trajectory of $\varphi $
(i.e., $\gamma(\mathbb T)$ is relatively compact and there exists
a point $y_0\in Y$ such that
$\gamma(t+s)=\varphi(t,\gamma(s),\sigma(s,y_0))$ for any $t\ge 0$
and $s\in\mathbb T$), then $\gamma (0)\in I_{y_0}$.
\end{coro}

\begin{theorem}\label{t2.7.5}\cite[ChII]{Che_2015}
Under the conditions of Theorem \ref{t2.7.3} $ w \in I_{y} \ ( y
\in Y ) $ if and only if there exits a whole trajectory $ \nu :
\mathbb T \to W $ of the cocycle $\varphi,$ satisfying the
following conditions: $ \nu (0) = w $ and $ \nu (\mathbb T ) $ is
relatively compact.
\end{theorem}

\begin{definition}\label{defF1} A family of subsets $\{I_{y}|\ y\in
Y\}$ ($I_y\subseteq W$ for any $y\in Y$) is said to be upper
semicontinuous if for any $y_0\in Y$ and $y_n\to y_0$ as $n\to
\infty$ we have
$$ \lim\limits_{n\to \infty}\beta(I_{y_n},I_{y_0})= 0.$$
\end{definition}

\begin{lemma}\label{lF1} The following statements hold:
\begin{enumerate}
\item the family $\{I_y|\ y\in Y\}$ is invariant if and only if
the set $J:=\bigcup \{J_y|\ y\in Y\}$, where $J_{y}:=I_{y}\times
\{y\}$, is invariant with respect to skew-product dynamical system
$(X,\mathbb T_{+},\pi)$ ($X:=W\times Y$ and $\pi
:=(\varphi,\sigma)$); \item if $\bigcup \{I_y|\ y\in Y\}$ is
relatively compact, then the family $\{I_y|\ y\in Y\}$ is upper
semicontinuous if and only if the set $J$ is closed in $X$.
\end{enumerate}
\end{lemma}
\begin{proof} The first statement is evident.

Suppose that the set $J\subseteq X$ is closed. If we suppose that
the family $\{I_{y}|\ y\in Y\}$ is not upper semicontinuous, then
there are $\varepsilon_0>0$, $y_0\in Y$ and sequences
$\{y_n\}\subset Y$ and $\{u_n\}\subset W $ such that $y_n\to y_0$
as $n\to \infty$, $u_n\in I_{y_n}$ and
\begin{equation}\label{eqF1}
\rho(u_n,I_{y_0})\ge \varepsilon_{0}.
\end{equation}
Since $\bigcup \{I_y|\ y\in Y\}$ is relatively compact, then
without loss of generality we can suppose that the sequence
$\{u_n\}$ is convergent. Denote by $u_0:=\lim\limits_{n\to
\infty}u_n$ and passing into limit as $n\to \infty$ in inequality
$(\ref{eqF1})$ we obtain $u_0\notin I_{y_0}$. On the other hand we
have $(u_n,y_n)\in J_{y_n}\subseteq J$ for any $n\in \mathbb N$
and since the set $J$ is closed and $(u_n,y_n)\to (u_0,y_0)$ as
$n\to \infty$, then $(u_0,y_0)\in J$ and, consequently, $u_o\in
I_{y_0}$. The obtained contradiction proves our statement.

Let now the family $\{I_{y}\}$ upper semicontinuous and
$(\bar{u},\bar{y})\in \overline{J}$. Then there exists a sequence
$\{(u_n,y_n)\in J\}$ such that $(u_n,y_n)\to (\bar{u},\bar{y})$.
Since $u_n\in I_{y_n}$ and $\{I_{y}|\ y\in Y\}$ is upper
semicontinuous, then $u_0\in I_{y_0}$ and, consequently,
$(u_0,y_0)\in J_{y_0}\subseteq J$. Thus the set $J$ is closed.
\end{proof}

\section{I. U. Bronshtein's conjecture for non-autonomous dynamical
systems}\label{S4}

\begin{example}\label{exS1*} {\em (\emph{Bebutov's dynamical system})
Let $X,W$ be two metric space. Denote by $C(\mathbb T\times W,X)$
the space of all continuous mappings $f:\mathbb T\times W\mapsto
X$ equipped with the compact-open topology and $\sigma$ be the
mapping from $\mathbb T\times C(\mathbb T\times W,X)$ into
$C(\mathbb T\times W,X)$ defined by the equality
$\sigma(\tau,f):=f_{\tau}$ for all $\tau\in\mathbb T$ and $f\in
C(\mathbb T\times W,X)$, where $f_{\tau}$ is the
$\tau$-translation (shift) of $f$ with respect to variable $t$,
i.e., $f_{\tau}(t,x)=f(t+\tau,x)$ for all $(t,x)\in \mathbb
T\times W$. Then \cite[Ch.I]{Che_2015} the triplet $(C(\mathbb
T\times W,X),\mathbb T,\sigma)$ is a dynamical system on
$C(\mathbb T\times W,X)$ which is called a \emph{shift dynamical
system} (\emph{dynamical system of translations} or
\emph{Bebutov's dynamical system}).

Recall that the function $\varphi\in C(\mathbb T,\mathbb R^{d})$
(respectively, $f\in C(\mathbb T \times \mathbb R^{d},\mathbb
R^{n})$) possesses the property $(A)$, if the motion
$\sigma(\cdot,\varphi)$ (respectively, $\sigma(\cdot,f)$)
generated by the function $\varphi$ (respectively, $f$) possesses
this property in the dynamical system $(C(\mathbb T,\mathbb
R^{d}), \mathbb T,\sigma)$ (respectively, $(C(\mathbb T\times
\mathbb R^{d},\mathbb R^{d}), \mathbb T,\sigma)$).

In the quality of the property $(A)$ there can stand stability in
the sense of Lagrange (st. $L$), uniform stability (un. st.
$\mathcal L^+$) in the sense of Lyapunov, periodicity, almost
periodicity, asymptotical almost periodicity and so on.

For example, a function $f\in C(\mathbb{T}\times \mathbb
R^{d},\mathbb R^{d})$ is called almost periodic (respectively,
recurrent etc) in $t\in \mathbb{R}$ uniformly with respect to
(w.r.t.) $w$ on every compact subset from $\mathbb R^{d}$, if the
motion $\sigma(\cdot,f)$ is almost periodic (respectively,
recurrent) in the dynamical system $(C(\mathbb{T}\times \mathbb
R^{d},\mathbb R^{d}),\mathbb{T},\sigma)$.}
\end{example}

\textbf{I. U. Bronshtein's conjecture for cocycles.} Suppose that
$\langle W,\varphi,(Y,\mathbb T,\sigma)\rangle$ is a cocycle under
$(Y,\mathbb T,\sigma)$ with the fiber $W$ and the following
conditions are fulfilled:
\begin{enumerate}
\item the cocycle $\varphi$ admits a compact global attractor
$\textbf{I}:=\{I_{y}|\ y\in Y\}$; \item the cocycle $\varphi$ is
positively uniformly Lyapunov stable, i.e., for any $\varepsilon
>0$ and nonempty compact subset $K\subseteq W$ there exists a
positive number $\delta =\delta(\varepsilon,K)$ such that
$\rho(u_1,u_2)<\delta$ ($u_1,u_2\in K$) implies
$\rho(\varphi(t,u_1,y),\varphi(t,u_2,y))<\varepsilon$ for any
$t\ge 0$ and $y\in Y$;  \item the dynamical system $(Y,\mathbb
T,\sigma)$ is minimal and Bohr almost periodic.
\end{enumerate}
Then for any $y\in Y$ there exists at least one point $w_{y}\in
I_{y}$ such that the motion $\varphi(t,u_y,y)$ is defined on
entire axis $\mathbb T$ and it is Bohr almost periodic.

One of the main goal of this paper is a positive answer to I. U.
Bronshtein's conjecture for monotone nonautonomous dynamical
systems (Corollary \ref{corPM1}).

Let $\langle \mathbb R^d,\varphi,(Y,\mathbb T,\sigma)\rangle$ be a
cocycle over $(Y,\mathbb T,\sigma)$ with fiber $\mathbb R^d$.

\begin{theorem}\label{thBM1} Assume that the cocycle $\varphi$
\begin{enumerate}
\item is monotone; \item admits a compact global attractor
$\textbf{I}:=\{I_{y}|\ y\in Y\}$; \item is positively uniformly
Lyapunov stable.
\end{enumerate}
Then the following statements hold:
\begin{enumerate}
\item if $y_0\in \omega_{y_0}$, then there exists a point
$a_{y_0}\in I_{y_0}$ such that the full trajectory $\gamma_{y_0}$
with $x_0:=\gamma_{y_0}(0)=(a_{y_{0}},y_0)$ is comparable by
character of recurrence with $y _0$; \item if $y_0$ is strongly
Poisson stable, then the exists a point $a_{y_0}\in I_{y_0}$ such
that the full trajectory $\gamma_{y_0}$ with
$x_0:=\gamma_{y_0}(0)=(a_{y_{0}},y_0)$ is strongly comparable by
character of recurrence with $y _0$.
\end{enumerate}
\end{theorem}
\begin{proof}
Let $x_0=(u_{y_0},y_0)$, where $u_{y_0}$ is an arbitrary point
from $W$. Since the cocycle $\varphi$ is compact dissipative, then
the semitrajectory
$\Sigma^{+}_{x_0}:=\{(\varphi(t,u_{y_0},y_{y_0}),\sigma(t,y_0))|\
t\in\mathbb T_{+}\}$ is conditionally precompact. It easy to check
that under the conditions of Theorem the conditions $(C1)-(C3)$
are fulfilled. By Theorem \ref{thM1} there exists at least one
point $a_{y_0}\in W$ such that the full trajectory $\gamma_{y_0}$
with $\gamma_{y_0}(0)=(a_{y_{0}},y_0)$ is comparable by character
of recurrence with $y _0$. According to Corollary \ref{corF1}
$a_{y_0}\in I_{y_0}$.

The second statement of Theorem can be proved using the same
argument as in the proof of the first statement but instead of
Theorem \ref{thM1} we need to apply Theorem \ref{thM2}.
\end{proof}

\begin{coro}\label{corPM1} Under the conditions of Theorem
\ref{thBM1} the following statements take place:
\begin{enumerate}
\item if the point $y_0$ is $\tau$-periodic (respectively, Levitan
almost periodic, almost recurrent, almost automorphic, recurrent,
Poisson stable), then there exist a point $a_{y_0}\in I_{y_0}$
such that the full trajectory $\gamma_{y_0}=(a_{y_0},y_{0})$ is
so; \item if the point $y_0$ is $\tau$-periodic (respectively,
quasi periodic, Bohr almost periodic, almost automorphic,
recurrent, pseudo recurrent and Lagrange stable), then there exist
a point $a_{y_0}\in I_{y_0}$ such that the point
$x_0:=(a_{y_0},y_0)$ is so.
\end{enumerate}
\end{coro}
\begin{proof}
This statement follows from the Theorems \ref{thBM1},
\ref{thShc2.1} and \ref{tS2}.
\end{proof}

\begin{remark}\label{remBM1} Corollary \ref{corPM1}
give as the positive answer for I. U. Bronshtein's conjecture for
monotone Bohr almost periodic systems.
\end{remark}

\section{Applications}\label{S5}

\subsection{Dissipative Cocycles}

Let $Y $ be a complete metric space (generally speaking
noncompact), $ \langle \mathbb R^d, \varphi,$ $(Y, \mathbb
T,\sigma)\rangle $ be a cocycle on the state space $\mathbb R^d$
and $(X,\mathbb T _+,\pi )$ be the corresponding skew-product
dynamical system, where $X:=\mathbb R ^d\times Y$ and $\pi
:=(\varphi ,\sigma)$.

\begin{definition}\label{def9.5.1} The cocycle
$ \langle \mathbb R^d, \varphi, (Y, \mathbb T,\sigma)\rangle $ is
said to be dissipative if for any $y \in Y $ there is a positive
number $r_y$ such that
$$
\limsup \limits _{t \to +\infty} \vert \varphi (t,u,y)\vert <r_y
$$
for any $y \in Y$ and $u\in \mathbb R^d,$ i.e., for all $u\in
\mathbb R ^d $ and $ y \in Y$ there exists a positive number
$L(u,y)$ such that $
 \vert \varphi (t,u,y)\vert <r_y
$ for any $t \ge L(u,y).$
\end{definition}

\begin{definition}\label{def9.5.2}
The cocycle $ \langle E, \varphi, (, \mathbb T,\sigma)\rangle $ is
said to be uniformly dissipative\index{uniformly dissipative} if
there exists a positive number $r$ ($r$ is not depend upon $y \in
Y$) such that for any $R>0$ there is a positive number $L(R)$ such
that $ \vert \varphi (t,u,y)\vert <r $ for all $y \in Y$ and
$\vert u \vert \le R$ and $ t \ge L(R).$
\end{definition}

\begin{remark}\label{remGA1} 1. If the space $E$ is
finite-dimensional, then the uniformly dissipative is compactly
dissipative.

2. If the space $E$ is finite-dimensional ($E=\mathbb R^d$) and
$Y$ is compact, then the dissipative system is uniformly
dissipative \cite[ChII]{Che_2015}.
\end{remark}

\begin{theorem}\label{t3.1.1}\cite[ChIII]{Che_2015} Let $Y$ be a compact and
$\langle R^{d}, \varphi, (Y,\mathbb{T},\sigma)\rangle $ be a
cocycle over the dynamical system $(Y,\mathbb{T},\sigma)$ with the
fiber $\mathbb R^d$. Then the following statements are equivalent:
\begin{enumerate}
\item[1.] There exists a positive number $R$ such that
$$
\limsup\limits_{t\to+\infty}\vert\varphi(t,u,y)\vert<R
$$
for all $u\in \mathbb R^d$ ¨ $y\in Y$. \item[2.] There is a
positive number $r_1$ such that for all $u\in \mathbb R^d$ and
$y\in Y$ there exists $\tau=\tau(u,y)>0$ for which
$\vert\varphi(\tau,u,y)\vert<r_1$. \item[3.] There is a positive
number $r_2$ such that
$$
\liminf\limits_{t\to+\infty}\vert\varphi(t,u,y)\vert<r_2
$$
for all $u\in \mathbb R^d$ and $y\in Y$. \item[4.] There exists a
positive number $R_0$ and for all $R>0$ there is $l(R)>0$ such
that $\vert\varphi(t,u,y)\vert\le R_0$ for all $t\ge l(R)$, $u\in
\mathbb R^d$, $\vert u\vert\le R$ and $y\in Y$.
\end{enumerate}
\end{theorem}

Note that every condition 1.-4. that figures in Theorem
\ref{t3.1.1} is equivalent to the dissipativity of the
non-autonomous dynamical system $\langle (X,\mathbb{T}_{+},\pi),$
$ (Y,\mathbb{T},\sigma), h\rangle $ associated by the cocycle
$\langle \mathbb R^d, \varphi,$ $ (Y,\mathbb{T},\sigma)\rangle $
over $(Y,\mathbb{T},\sigma)$ with the fiber $\mathbb R^d$.

\subsection{Ordinary Differential Equations}

We will give below an example of a skew-product dynamical system
which plays an important role in the study of non-autonomous
differential equations.

\begin{example}\label{ex3.1.2}
{\rm  Consider the differential equation
\begin{equation}\label{eq3.1.1}
u'=f(t,u),
\end{equation}
where $f\in C(\mathbb{R}\times \mathbb R^d,\mathbb R^d)$. Along
with the equation (\ref{eq3.1.1}) we consider its $H$-class
\cite{bro75},\cite{Dem67}, \cite{Lev-Zhi},
\cite{scher72},\cite{Sch85}, i.e., the family of the equations
\begin{equation}
v'=g(t,v),\label{eq3.1.2}
\end{equation}
where $g\in H(f)=\overline{\{f_{\tau}:\tau\in \mathbb{R}\}}$ and
$f_{\tau}(t,u)=f(t+\tau,u)$, where the bar indicating closure in
the compact-open topology.

\textbf{Condition (A1)}. The function $f\in C(\mathbb{R}\times
\mathbb R^d,\mathbb R^d)$ is said to be regular if for every
equation (\ref{eq3.1.2}) the conditions of existence, uniqueness
and extendability on $\mathbb{R}_{+}$ are fulfilled.

We will suppose that the function $f$ is regular. Denote by
$\varphi(\cdot,v,g)$ the solution of (\ref{eq3.1.2}) passing
through the point $v\in \mathbb R^d$ for $t=0$. Then the mapping
$\varphi:\mathbb{R}_{+}\times \mathbb R^d\times H(f)\to \mathbb
R^d$ satisfies the following conditions (see, for example,
\cite{bro75},\cite{Sell67.1},\cite{Sell67.2}):
\begin{enumerate}
\item[$1)$] $\varphi(0,v,g)=v$ for all $v\in \mathbb R^d$ and
$g\in H(f)$; \item[$2)$]
$\varphi(t,\varphi(\tau,v,g),g_{\tau})=\varphi(t+\tau,v,g)$ for
each $ v\in \mathbb R^d$, $g\in H(f)$ and $t,\tau \in
\mathbb{R}_{+}$; \item[$3)$] $\varphi:\mathbb{R}_{+}\times
E^n\times H(f)\to \mathbb R^d$ is continuous.
\end{enumerate}

Denote by $Y:=H(f)$ and $(Y,\mathbb{R},\sigma)$ a dynamical system
of translations on $Y$, induced by the dynamical system of
translations $(C(\mathbb{R}\times \mathbb R^d,\mathbb
R^d),\mathbb{R},\sigma)$. The triple $\langle \mathbb R^d,\varphi,
(Y,\mathbb{R},\sigma)\rangle $ is a cocycle over
$(Y,\mathbb{R}_{+},\sigma)$ with the fiber $\mathbb R^d$. Hence,
the equation (\ref{eq3.1.1}) generates a cocycle $\langle \mathbb
R^d,\varphi, (Y,\mathbb{R},\sigma)\rangle $ and the non-autonomous
dynamical system $\langle (X,\mathbb{R}_{+},\pi),\,
(Y,\mathbb{R},\sigma), h\rangle $, where $X:= \mathbb R^d\times
Y$, $\pi:=(\varphi,\sigma)$ and $h:=pr_2:X\to Y$.}
\end{example}

\begin{definition}\label{def3.1.1}
Recall that the equation (\ref{eq3.1.1}) is called dissipative
\cite{Dem67}, \cite{Pliss66}, \cite{Yosh55}, \cite{Yosh59}, if for
all $t_0\in\mathbb{R}$ and $x_0\in E^n$ there exists a unique
solution $x(t;x_0,t_0)$ of the equation (\ref{eq3.1.1}) passing
through the point $(x_0,t_0)$ and if there exists a number $R>0$
such that $\lim\limits_{t\to+\infty}\sup\vert x(t;x_0,t_0)\vert<R$
for all $x_0\in \mathbb R^d$ and $t_0\in \mathbb{R}$. In other
words, for every solution $x(t;x_0,t_0)$ there is an instant
$t_1={t_0}+l(t_0,x_0)$, such that $\vert x(t;x_0,t_0)\vert<R$ for
any $t\ge t_1$. If for any $r>0$ the number $l(t_0,x_0)$ can be
chosen independent on $t_0$ and $x_0$ with $|x_0|\le r$, then the
equation (\ref{eq3.1.1}) is called uniformly dissipative
\cite{Dem67}.
\end{definition}

Below we will establish the relation between the dissipativity of
the equation\index{dissipativity of the equation} (\ref{eq3.1.1})
and the dissipativity of the non-autonomous dynamical system
generated by the equation (\ref{eq3.1.1}).

\begin{lemma}\label{lDE1} Suppose that the function $f\in C(\mathbb R\times \mathbb R^{d},\mathbb
R^{d})$ is regular. If equation (\ref{eq3.1.1}) is uniformly
dissipative, then the cocycle $\varphi$ generated by equation
(\ref{eq3.1.1}) is also uniformly dissipative.
\end{lemma}
\begin{proof}
Let (\ref{eq3.1.1}) be uniformly dissipative, then there exists a
positive number $R$ such that for any $r>0$ we can choose
$l=l(r)>0$ so that
\begin{equation}\label{eqUD1}
|x(t;t_0,x_0)|<R
\end{equation}
for any $t\ge t_0+l(r)$. If $g\in H(f)$, then there exists a
sequence $\{\tau_n\}\subset \mathbb R$ such that $f_{\tau_{n}}\to
g$. Since $\varphi(t,x_0,f_{\tau})=x(t+\tau;\tau,x_0)$ for any
$t\in\mathbb R,$ $\tau\in\mathbb R$ and $x_0\in\mathbb R^{d}$,
then we have
\begin{equation}\label{eqUD2}
|\varphi(t,x_0,f_{\tau_{n}})|=|x(t+\tau_{n};\tau_{n},x_0)|<R
\end{equation}
for any $t\ge 0$, $|x_0|\le r$ and $n\in\mathbb N$. Passing in
limit in (\ref{eqUD2}) as $n\to \infty$ we obtain
$|\varphi(t,v,g)|\le R$ for any $t\le l(r)$, $|x_0|\le r$ and
$t\ge l(r)$.
\end{proof}

\begin{lemma}\label{l3.1.3}\cite[ChIII]{Che_2015}  Let $f\in C(\mathbb{R}\times E^n,E^n)$
be regular. If $H(f)$ is compact, then equation (\ref{eq3.1.1}) is
uniformly dissipative if and if there is a positive number $r$
such that
\begin{equation}\label{eq3.1.4}
\limsup\limits_{t\to+\infty}\vert\varphi(t,x_0,g)\vert<r \quad
(x_0\in \mathbb R^d,\, g\in H(f))\ .\nonumber
\end{equation}
\end{lemma}

Thus, for the equation (\ref{eq3.1.1}) ($f$ is regular and $H(f)$
is compact) we established that it is uniformly dissipative if and
only if the non-autonomous dynamical system generated by this
equation is dissipative.

\textbf{Condition (A2).} Equation (\ref{eq3.1.1}) is monotone.
This means that the cocycle $\langle \mathbb R^n,\varphi,$ $
(H(f),$ $\mathbb{R},$ $\sigma)\rangle$ (or shortly $\varphi$)
generated by (\ref{eq3.1.1}) is monotone, i.e., if $u,v\in \mathbb
R^{d}$ and $u\le v$ then $\varphi(t,u,g)\le \varphi(t,v,g)$ for
all $t\ge 0$ and $g\in H(f)$.

Let $K$ be a closed cone in $\mathbb R^d$. The dual cone to K is
the closed cone $K^{*}$ in the dual space $\big{(}\mathbb
R^{d}\big{)}^{*}$ of linear functions on $\mathbb R^d$, defined by
\begin{equation}\label{eqK01}
K^{*}:=\{\lambda \in \big{(}\mathbb R^{d}\big{)}^{*}:\ \langle
\lambda,x\rangle \ge 0\ \mbox{for any}\ x\in K\},\nonumber
\end{equation}
where $\langle \cdot , \cdot\rangle$ is the scalar product in
$\mathbb R^{d}$.

Recall \cite{Smi_1987},\cite[ChV]{Smi_1995} that the function
$f\in C(\mathbb R\times \mathbb R^{d},\mathbb R^d)$ is said to be
quasimonotone if for any $(t,u),(t,v)\in \mathbb R\times \mathbb
R^d$ and $\phi \in \big{(}\mathbb R^{d}_{+}\big{)}^{*}$ we have:
$u \le v$ and $\phi(u)=\phi(v)$ implies $\phi(f(t,u))\le
\phi(f(t,v))$.

\begin{lemma}\label{lK1} Let $f\in
C(\mathbb R\times \mathbb R^{d},\mathbb R^d)$ be a regular and
quasimonotone function, then the following statements hold:
\begin{enumerate}
\item if $u\le v$, then $\varphi(t,u,f)\le \varphi(t,v,f)$ for any
$t\ge 0$; \item any function $g\in H(f)$ is quasimonotone; \item
$u\le v$ implies $\varphi(t,u,g)\le \varphi(t,v,g)$ for any $t\ge
0$ and $g\in H(f)$; \item equation (\ref{eq3.1.1}) is monotone.
\end{enumerate}
\end{lemma}
\begin{proof} The first statement is proved in
\cite[ChIII]{HS_2005}.

Let $g\in H(f)$, then there exists a sequence $\{h_k\}\subset
\mathbb R$ such that $g(t,u)=\lim\limits_{k\to
\infty}f_{h_k}(t,u)$ for any $(t,u)\in\mathbb R\times \mathbb
R^d$. Let $u\le v$ ($u,v\in\mathbb R^d$) and $\phi \in
\big{(}\mathbb R^{n}_{+}\big{)}^{*}$ such that $\phi(u)=\phi(v)$.
Since $f$ is quasimonote, then we will have
\begin{equation}\label{eqK1}
f_{i}(t+h_k,u)\le f_{i}(t+h_k,v)
\end{equation}
and passing into limit in (\ref{eqK1}) as $k\to \infty$ we obtain
that $g$ is quasimonotone too.

Finally, the third statement follows from the first and second
statements. Lemma is completely proved.
\end{proof}

\begin{definition}\label{defOO1*} A solution $\varphi(t,u_0,f)$ of
equation (\ref{eq3.1.1}) is said to be:
\begin{enumerate}
\item[-] uniformly Lyapunov stable in the positive direction, if
for arbitrary $\varepsilon >0$ there exists $\delta
=\delta(\varepsilon)>0$ such that
$|\varphi(t_0,u,f)-\varphi(t_0,u_0,f)|<\delta$ ($t_0\in \mathbb
R$, $u\in\mathbb R^d$) implies
$|\varphi(t,x,f)-\varphi(t,x_0,f)|<\varepsilon$ for any $t\ge
t_0$; \item[-] compact on $\mathbb R_{+}$ if the set
$Q:=\overline{\varphi(\mathbb R_{+},u_0,f)}$ is a compact subset
of $\mathbb R^{d}$, where by bar is denoted the closure in
$\mathbb R^{d}$ and $\varphi(\mathbb
R_{+},u_0,f):=\{\varphi(t,u_0,f):\ t\in \mathbb R_{+}\}$.
\end{enumerate}
\end{definition}

Let $f\in C(\mathbb R\times \mathbb R^d,\mathbb R^d)$,
$\sigma(t,f)$ be the motion (in the shift dynamical system
$(C(\mathbb R\times \mathbb R^d,\mathbb R^d),\mathbb R,\sigma)$)
generated by $f$, $u_0\in \mathbb R^d$, $\varphi(t,u_0,f)$ be the
solution of equation (\ref{eq3.1.1}), $x_0:=(u_0,f)\in X:=\mathbb
R^d\times H(f)$ and $\pi(t,x_0):=(\varphi(t,u_0,f),\sigma(t,f))$
the motion of skew-product dynamical system $(X,\mathbb
R_{+},\pi)$.

\begin{definition}\label{defCS01*} A solution $\varphi(t,u_0,f)$ of
equation (\ref{eq3.1.1}) is called
\cite{Che_2009},\cite{scher72},\cite{Sch85} compatible
(respectively, strongly compatible or uniformly compatible) if the
motion $\pi(t,x_0)$ is comparable (respectively, strongly
comparable or uniformly comparable) by character of recurrence
with $\sigma(t,f)$.
\end{definition}

\begin{lemma}\label{l3.1}\cite{CC_2009}
Let $\langle W,\varphi, (Y,\mathbb T,\sigma)\rangle$ be a cocycle
and $ \langle (X,\mathbb T_{+},\pi),(Y,\mathbb T,\sigma),h\rangle $
be a non-autonomous dynamical system associated by cocycle
$\varphi$. Suppose that $x_0:=(u_0,y_0)\in X:=W\times Y$ and the set
$Q_{(u_0,y_0)}^{+}:=\overline{\{\varphi(t,u_0,y_0)\ |\ t\in \mathbb
T_{+}\}},$ where $\mathbb T_{+}:=\{t\in \mathbb T\ |\ t\ge 0 \}$, is
compact.

Then the set $H^{+}(x_0):=\overline{\{\pi(t,x_0)\ |\ t\in\mathbb
T_{+}\}}$ is conditionally compact.
\end{lemma}

\begin{remark}\label{remAG} If $x_0:=(u_0,y_0)\in X:=W\times Y$ and
$\alpha_{y_0}$ (respectively, $\gamma_{y_0}$) is a point from $X$
defined in Lemma \ref{lAPS2} then we denote by $\alpha_{u_0}$
(respectively, $\gamma_{u_0}$) a point from $W$ such that
$\alpha_{y_0}=(\alpha_{u_0},y_0)$ (respectively,
$\gamma_{y_0}=(\gamma_{u_0},y_0)$).
\end{remark}

\begin{definition} A function $f$ is said to be Poisson stable (respectively, strongly Poisson stable) in $t\in\mathbb T$ uniformly
with respect to $u$ on every compact subset of $\mathbb R^d$ if
the point $f\in C(\mathbb T\times \mathbb R^d,\mathbb R^d)$ is
Poisson stable (respectively, strongly Poisson stable) in shift
dynamical system $ (C(\mathbb T\times \mathbb R^d,\mathbb
R^d),\mathbb T,\sigma)$.
\end{definition}

\begin{theorem}\label{thA1} Suppose that the following assumptions
are fulfilled:
\begin{enumerate}
\item[-] the function $f\in C(\mathbb R\times \mathbb R^d,\mathbb
R^nd)$ is positively Poisson stable in $t\in\mathbb R$ uniformly
with respect to $u$ on every compact subset from $\mathbb R^d$;
\item[-] equation (\ref{eq3.1.1}) is uniformly dissipative;
\item[-] each solution $\varphi(t,u_0,f)$ of equation
(\ref{eq3.1.1}) is positively uniformly Lyapunov stable.
\end{enumerate}

Then under conditions $(A1)-(A2)$ the following statement hold:
\begin{enumerate} \item[1.] equation (\ref{eq3.1.1}) has at least
one solution $\varphi(t,\gamma_{u_0},f)$ defined and bounded on
$\mathbb R$ which is compatible and belongs to Levinson center of
(\ref{eq3.1.1}). \item[2.] if the function $f\in C(\mathbb R\times
\mathbb R^n,\mathbb R^n)$ is stationary (respectively,
$\tau$-periodic, Levitan almost periodic, almost recurrent, almost
automorphic, Poisson stable) in $t\in \mathbb R$ uniformly with
respect to $u$ on every compact subset from $\mathbb R^n$, then
$\varphi(t,\gamma_{u_0},f)$ is also stationary (respectively,
$\tau$-periodic, Levitan almost periodic, almost recurrent, almost
automorphic, Poisson stable).
\end{enumerate}
\end{theorem}
\begin{proof} Let $f\in C(\mathbb R\times \mathbb R^d,\mathbb R^d)$
and $(C(\mathbb R\times \mathbb R^d,\mathbb R^nd),\mathbb
R,\sigma)$ be the shift dynamical system no $C(\mathbb R\times
\mathbb R^d,\mathbb R^d)$. Denote by $Y:=H(f)$ and $(Y,\mathbb
R,\sigma)$ the shift dynamical system on $H(f)$ induced by
$(C(\mathbb R\times \mathbb R^d,\mathbb R^d),\mathbb R,\sigma)$.
Consider the cocycle $\langle\mathbb R^d,\varphi,(Y,\mathbb
R,\sigma)\rangle$ generated by equation (\ref{eq3.1.1}) (see
Condition (A1)). Now to finish the proof of Theorem it is
sufficient to apply Theorems \ref{thBM1} (the first statement) and
Corollary \ref{corPM1}. Theorem is proved.
\end{proof}

\begin{theorem}\label{thA2} Suppose that the following assumptions
are fulfilled:
\begin{enumerate}
\item[-] the function $f\in C(\mathbb R\times \mathbb R^d,\mathbb
R^d)$ is strongly Poisson stale in $t\in\mathbb R$ uniformly with
respect to $u$ on every compact subset from $\mathbb R^d$;
\item[-] equation (\ref{eq3.1.1}) is uniformly dissipative;
\item[-] each solution $\varphi(t,u_0,g)$ of every equation
(\ref{eq3.1.2}) is positively uniformly Lyapunov stable.
\end{enumerate}

Then under conditions $(A1)-(A2)$ the following statements hold:
\begin{enumerate}
\item[1.] every equation (\ref{eq3.1.2}) has at least one solution
$\varphi(t,\gamma_{v_0},g)$ defined and bounded on $\mathbb R$
such that:
\begin{enumerate}
\item[2.] solution $\varphi(t,\gamma_{v_0},g)$ belongs to Levinson
center of equation (\ref{eq3.1.2}); \item[3.]
$\varphi(t,\gamma_{v_0},g)$ is a strongly compatible solution of
(\ref{eq3.1.2}).
\end{enumerate}
\item[4.] if the function $f\in C(\mathbb R\times \mathbb
R^d,\mathbb R^d)$ is stationary (respectively, $\tau$-periodic,
Bohr almost periodic, almost automorphic, recurrent, pseudo
recurrent and $H(f)$ is compact, uniformly Poisson stable and
$H(f)$ is compact) in $t\in \mathbb R$ uniformly with respect to
$u$ on every compact subset from $\mathbb R^d$, then
$\varphi(t,\gamma_{u_0},f)$ is also stationary (respectively,
$\tau$-periodic, Levitan almost periodic, almost recurrent, almost
automorphic, uniformly Poisson stable).
\end{enumerate}
\end{theorem}
\begin{proof} This statement can be proved similarly as Theorem
\ref{thA1} using  Theorems \ref{thBM1} (the second statement) and
Corollary \ref{corPM1}. Theorem is proved.
\end{proof}

\subsection{Difference Equations}

\begin{example}\label{ex2.1}
{\em Consider the equation
\begin{equation}\label{eq2.1}
u_{n+1}=f(n,u_n),
\end{equation}
where $f\in C(\mathbb Z \times \mathbb R^{d},\mathbb R^{d})$.
Along with equation (\ref{eq2.1}) we will consider $H$-class of
(\ref{eq2.1}), i.e., the family of equation
\begin{equation}\label{eq2.1*}
v_{n+1}=g(n,v_n), \ (g\in H(f))
\end{equation}
where $H(f):=\overline{\{f_{\tau}|\ \tau\in\mathbb Z\}}$ and by
bar is denoted the closure in the space $C(\mathbb Z\times \mathbb
R^{d},\mathbb R^{d})$.

Denote by $\varphi(n,v,g)$ the solution of equation (\ref{eq2.1*})
with initial condition $\varphi(0,v,g)=v.$ From the general
properties of difference equations it follows that:
\begin{enumerate}
\item $\varphi(0,v,g)=v$ for all $v\in \mathbb R^{d}$ and $g\in
H(f);$ \item $\varphi (n+m,v,g)=\varphi (n,\varphi
(m,v,g),\sigma(m,g))$ for all $n,m \in  \mathbb Z_{+}$ and
$(v,g)\in \mathbb R^{d}\times H(f)$; \item the mapping $\varphi$
is continuous.
\end{enumerate}

Thus every equation (\ref{eq2.1}) generate a cocycle $\langle
\mathbb R^{d},\varphi, (H(f), \mathbb Z,\sigma)\rangle$ over
$(H(f), \mathbb Z,\sigma)$ with the fiber $\mathbb R^{d}$.}
\end{example}

\begin{lemma}\label{l8.2} Let $f\in C(\mathbb Z\times \mathbb R^{d},\mathbb R^{d}).$ Suppose that the following conditions hold:
\begin{enumerate}
\item $u_1,u_2\in \mathbb R^{d}$ and $u_1\le u_2$; \item the
function $f$ is monotone non-decreasing with respect to variable
$u\in \mathbb R^{d},$ i.e., $f(t,u_1)\le f(t,u_2)$
 for any $t\in\mathbb Z$.
\end{enumerate}

Then $\varphi(n,v_1,g) \le \varphi(n,v_2,g)$ for any $n\in Z_{+}$,
$v_1,v_2\in\mathbb R^{d}$ with $v_1\le v_2$ and $g\in H(f)$.
\end{lemma}
\begin{proof}. Let $g\in H(f)$ and $v_1,v_2\in \mathbb R^{d}$ with
$v_1\le v_2$, then there exists a sequence $\{\tau_{k}\}\subset
\mathbb Z$ such that $g(t,v)=\lim\limits_{k\to
\infty}f(t+\tau_{k},v)$ uniformly w.r.t. $v$ on every compact
subset of $\mathbb R^{d}$. Since $f$ is monotone in $u$, then we
have
\begin{equation}\label{eqME1}
f(t+\tau_{k},v_{1})\le f(t+\tau_{k},v_{2})
\end{equation}
for any $t\in\mathbb Z$ and $k\in\mathbb N$. Passing in limit as
$k\to \infty$ in (\ref{eqME1}) we obtain $g(t,v_1)\le g(t,v_2)$
for any $t\in \mathbb Z$.

Let $v_1\le v_2$  and $g\in H(f)$, then we have
$$\varphi(1,v_1,g)=g(0,v_1)\le g(0,v_2)=\varphi(1,v_2,g).$$
Suppose that $\varphi(s,v_1,g) \le \varphi(s,v_2,g)$ for all
$ks\le t,$ then we obtain
$$\varphi(s+1,v_1,g)=g(1,\varphi(s,v_1,g))\le g(1,\varphi(s,v_2,g))=\varphi(s+1,v_2,g).
$$
\end{proof}

\textbf{Condition (D).} Equation (\ref{eq2.1}) is monotone. This
means that the cocycle $\langle \mathbb R^d,\varphi,$ $ (H(f),$
$\mathbb{Z},$ $\sigma)\rangle$ (or shortly $\varphi$) generated by
(\ref{eq2.1}) is monotone, i.e., if $u,v\in \mathbb R^{n}$ and
$u\le v$ then $\varphi(t,u,g)\le \varphi(t,v,g)$ for all $t\ge 0$
and $g\in H(f)$.

\begin{definition}\label{defOO1} A solution $\varphi(t,u_0,f)$ of
equation (\ref{eq2.1}) is said to be:
\begin{enumerate}
\item[-] uniformly Lyapunov stable in the positive direction, if
for arbitrary $\varepsilon >0$ there exists $\delta
=\delta(\varepsilon)>0$ such that
$|\varphi(t_0,u,f)-\varphi(t_0,u_0,f)|<\delta$ ($t_0\in \mathbb
Z$, $u\in\mathbb R^d$) implies
$|\varphi(t,x,f)-\varphi(t,x_0,f)|<\varepsilon$ for any $t\ge
t_0$; \item[-] compact on $\mathbb Z_{+}$ if the set
$Q:=\overline{\varphi(\mathbb Z_{+},u_0,f)}$ is a compact subset
of $\mathbb R^{d}$, where by bar is denoted the closure in
$\mathbb R^{n}$ and $\varphi(\mathbb
Z_{+},u_0,f):=\{\varphi(t,u_0,f):\ t\in \mathbb Z_{+}\}$.
\end{enumerate}
\end{definition}

Let $f\in C(\mathbb Z\times \mathbb R^d,\mathbb R^d)$,
$\sigma(t,f)$ be the motion (in the shift dynamical system
$(C(\mathbb Z\times \mathbb R^d,\mathbb R^d),\mathbb Z,\sigma)$)
generated by $f$, $u_0\in \mathbb R^n$, $\varphi(t,u_0,f)$ be the
solution of equation (\ref{eq2.1}), $x_0:=(u_0,f)\in X:=\mathbb
R^d\times H(f)$ and $\pi(t,x_0):=(\varphi(t,u_0,f),\sigma(t,f))$
the motion of skew-product dynamical system $(X,\mathbb
Z_{+},\pi)$.

\begin{definition}\label{defCS01} A solution $\varphi(t,u_0,f)$ of
equation (\ref{eq2.1}) is called
\cite{Che_2009},\cite{scher72},\cite{Sch85} compatible
(respectively, strongly compatible or uniformly compatible) if the
motion $\pi(t,x_0)$ is comparable (respectively, strongly
comparable or uniformly comparable) by character of recurrence
with $\sigma(t,f)$.
\end{definition}

\begin{theorem}\label{thD1} Suppose that the following assumptions
are fulfilled:
\begin{enumerate}
\item[-] the function $f\in C(\mathbb Z\times \mathbb R^d,\mathbb
R^d)$ is positively Poisson stable in $t\in\mathbb Z$ uniformly
with respect to $u$ on every compact subset from $\mathbb R^n$;
\item[-] equation (\ref{eq2.1}) is uniformly dissipative; \item[-]
each solution $\varphi(t,u_0,f)$ of equation (\ref{eq2.1}) is
positively uniformly Lyapunov stable.
\end{enumerate}

Then under condition $(D)$ the following statement hold:
\begin{enumerate} \item[1.] equation (\ref{eq2.1}) has at least
one solution $\varphi(t,\gamma_{u_0},f)$ defined and bounded on
$\mathbb Z$ which is compatible and belongs to Levinson center of
(\ref{eq2.1}). \item[2.] if the function $f\in C(\mathbb Z\times
\mathbb R^d,\mathbb R^d)$ is stationary (respectively,
$\tau$-periodic, Levitan almost periodic, almost recurrent, almost
automorphic, Poisson stable) in $t\in \mathbb Z$ uniformly with
respect to $u$ on every compact subset from $\mathbb R^d$, then
$\varphi(t,\gamma_{u_0},f)$ is also stationary (respectively,
$\tau$-periodic, Levitan almost periodic, almost recurrent, almost
automorphic, Poisson stable).
\end{enumerate}
\end{theorem}
\begin{proof} Let $f\in C(\mathbb Z\times \mathbb R^d,\mathbb R^d)$
and $(C(\mathbb Z\times \mathbb R^d,\mathbb R^d),\mathbb
Z,\sigma)$ be the shift dynamical system no $C(\mathbb Z\times
\mathbb R^d,\mathbb R^d)$. Denote by $Y:=H(f)$ and $(Y,\mathbb
Z,\sigma)$ the shift dynamical system on $H(f)$ induced by
$(C(\mathbb Z\times \mathbb R^d,\mathbb R^d),\mathbb Z,\sigma)$.
Consider the cocycle $\langle\mathbb R^n,\varphi,(Y,\mathbb
Z,\sigma)\rangle$ generated by equation (\ref{eq2.1}). Now to
finish the proof of Theorem it is sufficient to apply  Theorems
\ref{thBM1} (the first statement) and Corollary \ref{corPM1}.
Theorem is proved.
\end{proof}

\begin{theorem}\label{thD2} Suppose that the following assumptions
are fulfilled:
\begin{enumerate}
\item[-] the function $f\in C(\mathbb Z\times \mathbb R^d,\mathbb
R^d)$ is strongly Poisson stale in $t\in\mathbb Z$ uniformly with
respect to $u$ on every compact subset from $\mathbb R^d$;
\item[-] equation (\ref{eq2.1}) is uniformly dissipative; \item[-]
each solution $\varphi(t,u_0,g)$ of every equation (\ref{eq2.1*})
is positively uniformly Lyapunov stable.
\end{enumerate}

Then under condition $(D)$ the following statements hold:
\begin{enumerate}
\item[1.] every equation (\ref{eq2.1*}) has at least one solution
$\varphi(t,\gamma_{v_0},g)$ defined and bounded on $\mathbb Z$
such that:
\begin{enumerate}
\item[2.] solution $\varphi(t,\gamma_{v_0},g)$ belongs to Levinson
center of equation (\ref{eq2.1}); \item[3.]
$\varphi(t,\gamma_{v_0},g)$ is a strongly compatible solution of
(\ref{eq2.1*}).
\end{enumerate}
\item[4.] if the function $f\in C(\mathbb Z\times \mathbb
R^d,\mathbb R^d)$ is stationary (respectively, $\tau$-periodic,
Bohr almost periodic, almost automorphic, recurrent, pseudo
recurrent and $H(f)$ is compact, uniformly Poisson stable and
$H(f)$ is compact) in $t\in \mathbb Z$ uniformly with respect to
$u$ on every compact subset from $\mathbb R^n$, then
$\varphi(t,\gamma_{u_0},f)$ is also stationary (respectively,
$\tau$-periodic, Levitan almost periodic, almost recurrent, almost
automorphic, uniformly Poisson stable).
\end{enumerate}
\end{theorem}
\begin{proof} This statement can be proved similarly as Theorem
\ref{thD1} using  Theorems \ref{thBM1} (the second statement) and
Corollary \ref{corPM1}. Theorem is proved.
\end{proof}

\end{document}